\documentclass[11pt,twoside, a4paper, english, reqno]{amsart}
\usepackage{graphicx, amsmath, varioref, amscd, amssymb,color, bm, stmaryrd, amsthm, epsfig, color, epic}
\usepackage{fancybox}
\usepackage{enumerate}
\usepackage[pdftex, colorlinks=true]{hyperref}
\usepackage{upref}
\usepackage{pdfsync}

\usepackage{eso-pic}
\usepackage{color}
\usepackage{type1cm}
\makeatletter
  \AddToShipoutPicture*{%
    \setlength{\@tempdimb}{.12\paperwidth}%
    \setlength{\@tempdimc}{.5\paperheight}%
    \setlength{\unitlength}{1pt}%
    \put(\strip@pt\@tempdimb,\strip@pt\@tempdimc){%
      \makebox(0,0){\rotatebox{90}{\textcolor[gray]{0.75}{\fontsize{2cm}{2cm}}}}
    }
}

\makeatother

\numberwithin{equation}{section}
\allowdisplaybreaks

\newtheorem{theorem}{Theorem}[section]

\theoremstyle{definition}
\newtheorem{definition}[theorem]{Definition}

\theoremstyle{remark}
\newtheorem{remark}[theorem]{Remark}

\newcommand{\Div}{\operatorname{div}}

\newcommand{\Grad}{\nabla}
\newcommand{\vr}{\varrho}

\newcommand{\vc}[1]{{\bm{#1}}}

\newcommand{\R}{\mathbb{R}}

\newcommand{\e}{\varepsilon}

\begin{document}

\title[On a tumor growth model] {On the vanishing viscosity approximation of a nonlinear model for tumor growth}

\author[Donatelli]{Donatella Donatelli}
\address[Donatelli]{\newline
Departement of Engineering Computer Science and Mathematics\\
University of L'Aquila\\
67100 L'Aquila, Italy.}
\email[]{\href{donatella.donatelli@univaq.it}{donatella.donatelli@univaq.it}}
\urladdr{\href{http://univaq.it/~donatell}{univaq.it/\~{}donatell}}

\author[Trivisa]{Konstantina Trivisa}
\address[Trivisa]{\newline
Department of Mathematics \\ University of Maryland \\ College Park, MD 20742-4015, USA.}
\email[]{\href{http://www.math.umd.edu}{trivisa@math.umd.edu}}
\urladdr{\href{http://www.math.umd.edu/~trivisa}{math.umd.edu/\~{}trivisa}}

\date{\today}



\thanks{}

\maketitle

\begin{abstract}
We investigate the dynamics of a nonlinear system modeling tumor growth with drug application. The tumor is viewed as a mixture consisting of proliferating, quiescent and dead cells as well as a nutrient  in the presence of a drug. 
The system is given by a multi-phase flow model: the densities of the different cells are governed by a set of  transport equations, the density of the nutrient and the density of the drug are governed by rather general diffusion equations, while the velocity of the tumor is given by Darcy's equation. The domain occupied by the tumor in this setting is a growing continuum  $\Omega$ with boundary $\partial \Omega$  both of which  evolve in time. Global-in-time weak solutions are obtained using an  approach based on the vanishing viscosity of the Brinkman's regularization. Both the solutions and the domain are rather general, no symmetry assumption is required and the result holds for large initial data. 
\end{abstract}

\tableofcontents{}

\newpage

\section{Introduction}\label{S1}

\subsection{Motivation}
In recent years, there has been an increased interest in the mathematical modeling and numerical simulation of tumor growth to complement experimental and clinical studies and thereby improve the understanding of cancer progression. 
Mathematical models describing continuum cell populations and their development typically consider the interactions between the cell number densities and one or more chemical species that provide nutrients and drug or influence the cell cycle events of a tumor cell population.

In this work we investigate the dynamics of a nonlinear system describing the evolution of cancerous cells.  The tumor is viewed as a {\em multiphase flow} consisting of proliferating cells, quiescent cells and dead cells (also known as extra-cellular cells) in the presence of a nutrient (oxygen) and drug. Here, and in what follows, we denote by $P, Q$ and $D$ the densities of proliferating, quiescent and dead cells respectively, and by $C$ and $W$ the nutrient and drug concentrations.

The mathematical model under consideration  is governed by  a system of transport equations for the evolution of cancerous cells; two rather general diffusion equations which are used to describe the diffusion  of the nutrient (oxygen) within the tumor region and  the evolution of the drug within the same regime and  the Darcy law,  which determines the velocity field. The  continuous movement within the tumor region is  due to proliferation, mitosis, apoptosis or removal of cells. 

\iffalse
 In general, these equations obey Fick's law: the nutrient is consumed at a rate proportional to the rate of cell mitosis, whereas the drug is consumed at a rate which is determined by the drug effectiveness,  
\fi

\subsection{Biological principles}
Our model is based on  the following biological principles (cf. Roda {\em et al.}  \cite{Roda-etal-2011, Roda-etal-2012}, Friedman {\em et al.} \cite{Friedman-2004}, \cite{ChenFriedman-2013}, Zhao  \cite{Zhao-2010}): 

\begin{enumerate}[\quad$\bullet$]
\item Living cells are either in a   {\em proliferating phase} or in a  {\em quiescent phase}.

\item Proliferating cells die as a result of {\em apoptosis,} which is a cell-loss mechanism. Quiescent cells die in part due to {\em apoptosis} and more often  due to starvation. In fact the proliferation and the necrotic death rates of tumor cells depend on the oxygen level.
\item The dead tumor cells are obtained from necrosis and apoptosis of live tumor cells, and they are cleared by macrophages.

\item Living cells undergo {\em mitosis}, a process that takes place in the nucleus of a dividing cell. 
\item Cells change from quiescent phase into proliferating phase at a rate which increases with the nutrient level, and they die at a rate which increases as the level of nutrient (oxygen)  decreases.
\item  Proliferating cells become quiescent   and die at a rate which increases as the nutrient concentration decreases. The proliferation rate increases with the nutrient concentration.
\item Proliferating cells and quiescent cells become dead cells at a rate which depends on the drug concentration.
\end{enumerate}
  
We denote by $\Omega_t :=\Omega(t)$  the tumor region and its boundary  $\partial\Omega_t$ evolves with respect to time. Both live and dead tumor cells are assumed to be in the tumor region $\Omega_t$. Abnormal proliferation of tumor cells generates internal pressure in $\Omega(t)$, resulting to a velocity field $\vc{v} \not= 0$.

%The tumor region $\Omega_t :=\Omega(t)$ is contained in a fixed domain $B$ and the region 
%$B\setminus \Omega_t$ represents the healthy tissue (see \figurename~\ref{regions}).  The tumor region 
%$\Omega_t$ and its boundary $\partial\Omega_t$ evolve with respect to time. Both live and dead tumor cells are assumed to be in the tumor region $\Omega_t;$ oxygen molecules can diffuse throughout the whole domain $B$.  Abnormal proliferation of tumor cells generates internal pressure in $\Omega(t)$, resulting to a velocity field $\vc{v} \not= 0$ (while $\vc{v} = 0$ in $B\setminus \Omega_t$).

%Global-in-time weak solutions are obtained using an  approach based on penalization of the boundary behavior, diffusion and  viscosity in the weak formulation. Both the solutions and the domain are rather general, no symmetry assumption is required and the result holds for large initial data. 

%\begin{figure}[htbp] 
%%\begin{center}
%\centering
%\includegraphics[width=5cm]{tumor.pdf}
%\caption{Healthy tissue - Tumor regime.} 
%\label{regions} 
%%\end{center}
%\end{figure}

\subsection{Governing equations of cells, oxygen and drug}
\subsubsection{Transport equations for the evolution of the cell densities}
All the cells are assumed to follow the general continuity equation:
\begin{equation}
\frac{\partial n}{\partial t} + \Div_x (n  {\vc{v}}) = G(n), \nonumber
\end{equation}
where $n$ may represent  densities of proliferating/quiescent and dead cells. The function $G$ includes in general proliferation, apoptosis or clearance of cells, and chemotaxis terms as appropriate.

The change of phase of the cancerous cells generates a continuous movement 
within the tumor represented by a velocity field ${\vc{v}}$. 
%We assume that there are three types of cells: proliferative cells with density $P,$ quiescent cells with density $Q$ and dead cells with density  $D$ in the presence of a 
%nutrient (oxygen) with density $C$ and a drug with density $W.$

The rates of change from one phase to another are functions of the nutrient concentration $C$.
\begin{enumerate}[\quad$\bullet$]
\item $K_Q (C)$ denotes the rate of change of phase from $ P \to Q;$
\item $K_P(C)$ denotes the rate of change  from $Q \to P;$
while 
\item $K_A  (C)$ and $K_D (C)$ denote the change of phases from $ P \to D$ and 
$Q \to D$ respectively.
\end{enumerate}
Here,  $K_A$ stands for apoptosis, whereas dead cells are removed at rate $K_R$ (independent of $C$), and the rate of cell proliferation (new births) is $K_B©.$ 
\smallskip

\subsubsection{The tumor tissue as a porous medium}
Due to proliferation and removal of cells there is continuous motion of cells within the tumor; this movement is represented by the velocity field $\vc{v}$ given by the Darcy's equation
\begin{equation}
\Grad_x \sigma = - \frac{\tilde{\mu}}{K} \vc{v}.\label{pressure2}
\end{equation}
%\vspace{0.08in}

\noindent
where $\sigma$ denotes the pressure, $\tilde{\mu}$  is a positive constant describing the viscous like properties of tumor cells, whereas $K$ denotes the permeability. 
In the present context,  \eqref{pressure2} accounts for  the friction of the tumor cells
with the extracellular matrix

The mass conservation laws for the densities of the proliferative cells $P,$ quiescent cells $Q$ and dead cells $D$  in $\Omega(t)$ 
take the following form:
\begin{equation}
 \frac{\partial P}{\partial t} + \Div (P \vc{v}) = \vc{G_P}, \label{dP}
 \end{equation}
 \begin{equation}
 \frac{\partial Q}{\partial t} + \Div (Q \vc{v}) = \vc{G_Q}, \label{dQ}
\end{equation}
\begin{equation}
\frac{\partial D}{\partial t} + \Div (D \vc{v}) = \vc{G_D}. \label{dD}
\end{equation}
%Without loss of generality we assume (cf.\  Friedman \cite{Friedman-2004}) that the source terms ${\bf \{G_P, G_Q, G_D\}}$ are of the form:
Following Friedman \cite{Friedman-2004}, the source terms ${\bf \{G_P, G_Q, G_D\}}$ are of the following form:
\begin{equation}\label{Gp}
 \vc{G_P} =  \left(K_B C - K_Q (\bar C- C) - K_A(\bar C - C)\right) P + K_P C Q - i_1 G_1(W) P, 
\end{equation}
where  $G_1(\cdot)$  a  smooth function and $K_{B}$, $K_{Q}$, $K_{A}$ are positive constants. The first term in this equation accounts for the increase of the number of cells due to new births, loss due to change of phase from proliferating to quiescent and loss due to apoptosis. The second term reflects the increase of the number of proliferating cells generated from quiescent cells, whereas the third term accounts for the decrease of the number of cells due to death resulting from the  effect of drug. In an analogous fashion
\begin{equation}\label{Gq}
 \vc{G_Q} = K_Q (\bar C- C)P - \left(K_P C + K_D(\bar C - C)\right) Q - i_2 G_2(W) Q,
\end{equation}
with $G_2(\cdot)$ a  smooth function and $K_{P}$, $K_{Q}$, $K_{D}$  positive constants. In the above relations \eqref{Gp}-\eqref{Gq} $i_1 G_1(W)$ and   $i_2 G_2(W)$ denote the rates by which  the proliferating cells and the quiescent  cells  become dead cells due to the drug. Finally,
\begin{equation}\label{Gd}
\vc{G_D} =  K_A (\bar C- C)P  + K_D (\bar C-C) Q - K_R D + i_1G_1(W)P + i_2 G_2(W)Q.
\end{equation}

\iffalse
\begin{equation}
\begin{cases} \label{GG}
% \label{G}
& \!\!\! \vc{G_P} =  \left(K_B C - K_Q (\bar C- C) - K_A(\bar C - C)\right) P + K_P C Q - i_1 G_1(W) P\\
& \!\!\! \vc{G_Q} = K_Q (\bar C- C)P - \left(K_P C + K_D(\bar C - C)\right) Q - i_2 G_2(W) Q \\
& \!\!\! \vc{G_D} =  K_A (\bar C- C)P  + K_D (\bar C-C) Q - K_R D + i_1G_1(W)P + i_2 G_2(W)Q,
\end{cases}
\end{equation}
\fi
\smallskip
\subsubsection{A linear diffusion equation for the evolution of nutrient}

Tumor cells consume nutrients (oxygen). In contrast to the equations of cell densities, the equations of the oxygen molecules in the tumor include diffusion terms in the following form:

\begin{equation}
\frac{\partial C}{\partial t} = \Grad \cdot (\nu_1 \Grad C) -  \left(K_1 K_P CP + K_2 K_Q(\bar{C}-C)Q\right)C. \nonumber
\end{equation}
Assuming that $\nu_1$ is constant this equation (cf. Friedman \cite{Friedman-2004}) becomes 
\begin{equation}\label{dC}
%\begin{cases}
\frac{\partial C}{\partial t} = \nu_1 \Delta C -  \left(K_1 K_P CP + K_2 K_Q(\bar{C}-C)Q\right)C.
%C(x,t) = \bar{C} \,\, \mbox{on}\,\, \partial \Omega(t),\,\,\, C(x,0) = C_0(x) \,\, \mbox{in}\,\, \Omega(0).
%\end{cases}
\end{equation}
This equation describes the diffusion of the oxygen in the tumor region. According to (cf. Ward and King \cite{WardKing-1997}, \cite{WardKing-2003}) the nutrient is consumed at a rate proportional to the rate of cell mitosis, namely  the  second term 
% $$ \left(K_1 K_P CP + K_2 K_Q(C-\bar{C})Q\right)C$$
  on the right-hand side of  the first equation in \eqref{dC}.
  
  \subsubsection{A linear diffusion equation for the evolution of drug}
 The evolution of the drug concentration in the tumor is given by a diffusion equation of the form
  \begin{equation}
\frac{\partial W}{\partial t} = \Grad \cdot (\nu_2 \Grad W) -  \left(\mu_1 G_1(W)P + \mu_2  G_2(W) Q\right) W,\nonumber
\end{equation}
with $G_1(\cdot), G_2(\cdot)$ smooth functions.

Assuming that $\nu_2$ is constant this equation (cf. Zhao \cite{Zhao-2010}) becomes 
\begin{equation}\label{dW}
%\begin{cases}
\frac{\partial W}{\partial t} = \nu_2 \Delta W -  \left(\mu_1 G_1(W) P + \mu_2 G_2(W) Q \right)W.
%&W(x,t) = \bar{W} \,\, \mbox{on}\,\, \partial \Omega(t),\,\,\, W(x,0) = W_0(x) \,\, \mbox{in}\,\, \Omega(0).
%\end{cases}
\end{equation}
This equation describes the diffusion of the drug within the tumor region. The second term of the right-hand side of \eqref{dW} represents the drug consumption, the constants $\mu_1, \mu_2$ are two positive constants which can be viewed as a measure of the drug effectiveness.

\smallskip

The total density of the mixture is denoted by $\vr_f$ and is given by\\
\begin{equation}
 \vr_f = P+Q+D = Constant.
\label{density}
\end{equation}

Adding \eqref{dP}-\eqref{dD} and taking into consideration \eqref{density} we arrive at the following relation, which represents an additional constraint\\
 \begin{align}
\rho_{f}\Div{\vc{v}}= \vc{G_P} + \vc{G_Q}+\vc{G_D}
=
K_{B} C P-K_{R}D.
\label{divcon}
\end{align}
Our aim is to study the system \eqref{pressure2}-\eqref{divcon} in a spatial domain $\Omega_t$, with a boundary $\Gamma =\partial\Omega_t$ varying in time. 

%\subsection{The singular limit $\mu \to 0$ of  Brinkman regularization}

\subsubsection{Boundary behavior}
The boundary of the domain $\Omega_t$ occupied by the tumor is described by means of a given velocity $\vc{V}(t, x),$ where $t \ge 0$ and $x \in \R^3.$ More precisely, assuming $\vc{V}$ is regular, we solve the associated system of differential equations
\begin{equation}
\frac{d}{dt} \vc{X}(t, x) = \vc{V}(t, \vc{X}(t,x)), \,\, t > 0, \,\, \vc{X}(0,x) = x, \nonumber
\end{equation}
and set
\begin{equation}
\begin{cases}
\!\!\! \! & \Omega_{\tau} = \vc{X}(\tau, \Omega_0), \,\, \mbox{where}\,\, \Omega_0 \subset \R^3\,\, \mbox{is a given domain,}\\
\!\! \!\! & \Gamma_{\tau} = \partial \Omega_{\tau}, \,\, \mbox{and} \,\, Q_{\tau} = \left\{(t,x) | t \in (0,\tau), x\in  \Omega_{\tau}\right\}.
\end{cases}\nonumber
\end{equation}
Moreover, we assume that   
\begin{equation}
\Div_x \vc{V}(\tau, \cdot) = 0,  \label{incomp-bc}
\end{equation}
which by the transport theorem yields
$$ |\Omega_\tau| = |\Omega_0| \,\,\, \mbox{for any}\,\,\, \tau \ge 0.$$

%\smallskip

 The model is closed by giving boundary conditions on the (moving) tumor boundary $\Gamma_{\tau}.$ 
More precisely, we assume that the boundary $\Gamma_{\tau}$ is impermeable,  meaning
\begin{equation}
 (\vc{v} - \vc{V}) \cdot \vc{n}|_{\Gamma_{\tau}} = 0, \,\, \mbox{for any}\,\,\, \tau \ge 0. \label{BC1}
 \end{equation}
In addition, for {\em viscous} fluids, Navier proposed the boundary condition of the form
\begin{equation}
[\mathbb{S} \vc{n}]_{\mbox{tan}}|_{\Gamma_{\tau}} = 0, \label{BC2}
\end{equation}
with $\mathbb{S}$ denoting the viscous stress tensor which in this context is assumed to be determined through Newton's rheological law
$$
\mathbb{S} = \mu \Big( \Grad \vc{v} + \Grad^{\perp} \vc{v} - {2 \over 3} \Div
\vc{v} \mathbb{I} \Big) + \xi \Div \vc{v} \mathbb{I},
$$
where $\mu> 0$, $\xi \geq 0$ are respectively the shear and bulk
viscosity coefficients. Condition \eqref{BC2} namely says that the tangential component of the normal 
viscous stress vanishes on $\Gamma_{\tau}.$
\smallskip

%In biological settings, flow situations arise  for which the no-slip condition gives rise to significant difficulties without capturing the essence of the problem. Specific cases in which some form of slip is essential include moving boundary domains involving a fluid-fluid interface or free surface that intersects a solid  surface (cf.\ Dussan, 1979; de Gennes, 1985; Behr, 2004; Qian {\em et al.}, 2006).
% While it might be that, even at this scale, the no-slip condition  is valid for a fluid in contact with a solid surface, there appear to be conditions (for instance nano bubbles on the surface) that result in apparent slip, making the slip condition a useful macroscopic model (Celata {\em et al.,} 1996).
  
The concentrations of the nutrient and the drug on the boundary satisfy the conditions:
\begin{equation}
C(x,t)|_{\Gamma_t} = 0,\,\,\, W(x,t)|_{\Gamma_t} = 0 \label{BC3}.
\end{equation} 
In contrast to the case of {\em avascular tumors} where the nutrient typically diffuses within the tumor region through the boundary, here we assume that the diffusion of the nutrient  occurs through the vessels present in the area. 

Finally, the problem \eqref{dP}-\eqref{BC3} is supplemented by the initial conditions
\begin{equation}
\begin{cases}
& \!\!\!\!  P(0, \cdot) = P_0, \,\, Q(0, \cdot) = Q_0, \,\, D(0, \cdot) = D_0, \\
& \!\!\!\!  C(0, \cdot) = C_{0}\leq \bar{C}, \,\, W(0, \cdot) = W_{0} \,\,\, \text{in} \,\,\, \Omega_0.
\end{cases}
 \label{IC}
\end{equation}
The aim of this work is the establishment of the global existence of weak solutions to 
the nonlinear system \eqref{pressure2}-\eqref{dD}, \eqref{dC}-\eqref{dW}
for finite large initial data.

Related results on the mathematical analysis of cancer  models have been presented by Zhao \cite{Zhao-2010} based on the framework introduced  by Friedman {\em et al.}  \cite{Friedman-2004}, \cite{ChenFriedman-2013}. The analysis in \cite{Friedman-2004}, \cite{ChenFriedman-2013} yields existence and uniqueness of solution  to a related model in the radial symmetric case for a small time interval $[0,T].$ The analysis in \cite{Zhao-2010} treats a parabolic-hyperbolic free boundary problem and provides a unique global solution in the radially symmetric case. 
%In the forth mentioned articles  the tumor tissue is assumed to be a porous medium and the velocity field is  determined by  Darcy's Law $$\vc{v} = - \Grad_x \sigma \,\, \mbox{in} \,\, \Omega(t).$$
In  \cite{DT-MixedModel-2013, DT-DrugApplication-2015}, Donatelli and Trivisa establish the global existence of weak solutions to a nonlinear system modeling tumor growth in a general moving domain $\Omega_t \subset \R^3$ without any symmetry assumption and for finite large initial data. In that context, the nonliner system is governed by  transport equations \eqref{dP}-\eqref{dW}  for the evolution of cancerous cells, whereas the evolution of the velocity field $\vc{v}$ of the tumor growth is given, 
 by the Brinkman regularization of the Darcy Law, namely
$$
\Grad_x \sigma = - \frac{\tilde{\mu}}{K} \vc{v} + \mu \Delta \vc{v}.
%\label{pressure3}
$$
In the present article, we establish the global existence of  weak solutions to the nonlinear system \eqref{S}
\begin{equation}
\begin{cases}
\displaystyle{\Grad_x\sigma = - \frac{\tilde{\mu}}{K} \vc{v}},\\\\
\displaystyle{ \frac{\partial P}{\partial t} + \Div (P \vc{v}) = \vc{G_P}}, \\\\
 \displaystyle{  \frac{\partial Q}{\partial t} + \Div (Q \vc{v}) = \vc{G_Q}}, \\\\
   \displaystyle{ \frac{\partial D}{\partial t} + \Div (D \vc{v}) = \vc{G_D}}, \\\\
 \displaystyle{  \frac{\partial C}{\partial t} = \nu_1 \Delta C -  \left(K_1 K_P CP + K_2 K_Q(\bar{C}-C)Q\right)C},\\\\
 \displaystyle{ \frac{\partial W}{\partial t} = \nu_2 \Delta W -  \left(\mu_1 G_1(W) P + \mu_2 G_2(W) Q \right)W}.
 \end{cases}
\tag{{\bf S}}
 \label{S}
 \end{equation}
 on time dependent domains supplemented with the boundary conditions \eqref{BC1}, \eqref{BC2},  \eqref{BC3} and the initial data \eqref{IC}, by 
establishing rigorously  the vanishing viscosity limit  $\mu \to 0$ for the following system,
\begin{equation}\tag{${\bf S_\mu}$}
\begin{cases}
\displaystyle{\Grad_x \sigma_{\mu} = - \frac{\tilde{\mu}}{K} \vc{v} _{\mu}+ \mu \Delta \vc{v}_{\mu}},\\\\
 \displaystyle{\frac{\partial P_{\mu}}{\partial t} + \Div (P_{\mu} \vc{v}_{\mu}) = \vc{G_{P_{\mu}}}}, \\\\
  \displaystyle{\frac{\partial Q_{\mu}}{\partial t} + \Div (Q_{\mu} \vc{v}_{\mu}) = \vc{G_{Q_{\mu}}}}, \\\\
  \displaystyle{ \frac{\partial D_{\mu}}{\partial t} + \Div (D_{\mu} \vc{v}_{\mu}) = \vc{G_{D_{\mu}}}}, \\\\
 \displaystyle{ \frac{\partial C_{\mu}}{\partial t} = \nu_1 \Delta C_{\mu} -  \left(K_1 K_P C_{\mu}P_{\mu} + K_2 K_Q(\bar{C}-C_{\mu})Q_{\mu}\right)C_{\mu}},\\\\
\displaystyle{ \frac{\partial W_{\mu}}{\partial t} = \nu_2 \Delta W_{\mu} -  \left(\mu_1 G_1(W_{\mu}) P + \mu_2 G_2(W_{\mu}) Q_{\mu} \right)W_{\mu}}.
 \end{cases}
 \label{Smu}
 \end{equation}
  with the aid of a series of delicate estimates that enable us to treat the vanishing viscosity limit within the time-dependent kinematic boundary. The global existence of weak solutions to \eqref{S}  is established for general solutions, that is no symmetry assumption is required and for large initial data.

\subsection{General strategy}
The main ingredients of our strategy can be formulated as follows:

\begin{enumerate}[\quad $\bullet$]
\item Starting from the nonlinear system $\eqref{Smu}$ the procedure, outlined in Section \ref{S2} below,
provides a global weak solution $$\{P_\mu, Q_\mu, D_\mu, \vc{v}_{\mu}, C_{\mu}, W_{\mu}\}.$$ The next step of the investigation involves the derivation of delicate a priori 
bounds (uniform in $\mu$) within the time dependent kinematic boundary. 
In this part, the condition \eqref{incomp-bc} imposed on the boundary behavior  is critical.

\item  The  uniform bounds in $\mu$ will allow us to establish the necessary compactness in order to pass into the limit 
$\mu \to 0$ obtaining the global existence of the solutions of the original problem  $\eqref{S}$.
An important tool in the analysis is the use of the extension operator for Sobolev spaces, $E:W^{1,2}(\Omega_{\tau})\longrightarrow W^{1,2}(\R^{3})$, which is uniformly bounded with respect to $t\in [0,T].$ This operator allow us to deal with the moving domain $\Omega_{\tau}$ in the following sense: since the  limiting process takes place in a moving domain $\Omega_{\tau}$ it will be easier to perform the limit, if we  extend $\vc{v}_{\mu}$, $C_{\mu}$ and $W_{\mu}$, $P_{\mu}$,  $Q_{\mu}$, $D_{\mu}$ on the whole domain $\R^{3}$ by  setting them equal to zero outside the tumor domain. Then, since the domain $\Omega_{\tau}$ is regular at each time  the  extension operator  $E:W^{1,2}(\Omega_{\tau})\longrightarrow W^{1,2}(\R^{3})$ can be of use.

\end{enumerate}

\subsection{Outline}
The paper is organized as follows: Section \ref{S1} presents the motivation, modeling  and   introduces the necessary preliminary material.  Section \ref{S2}  provides  weak formulation of the problem \eqref{S} and states the main result. Section \ref{S3} presents an outline of the global existence of weak solutions of the nonlinear system \eqref{Smu}.
In Section \ref{S4}  we present delicate a priori  bounds  which yield the necessary compactness that is needed in order to perform rigorously the singular limit. In Section \ref{S5} the rigorous limit   $\mu \to 0$ is established and we complete the proof of our Main Theorem \ref{T2.2}.

\section{Weak formulation and main results}\label{S2}
In this section we present the notion of weak solutions to the nonlinear system $\eqref{S}$. 
\subsection{Weak solutions}
\begin{definition}\label{D2.1}
 We say that $(P, Q, D, \vc{v}, C, W)$ is a weak solution of problem $\eqref{S}$ supplemented with boundary data satisfying
(\ref{BC1})-(\ref{BC3})  and initial data $(P_0,  Q_0, D_0, C_0, W_0)$
satisfying (\ref{IC}) provided that the following hold:
\vspace{0.1in}

$\bullet$ $(P,Q, D) \ge 0$ represents a weak solution of \eqref{dP}-\eqref{dQ}-\eqref{dD} on $(0,\infty)\times\Omega_{\tau}$, i.e., for any test function $\varphi \in C^{\infty}_c (([0,T)\times \mathbb{R}^3), T>0$ 
the  following integral relations hold
\smallskip

%\begin{equation} 
%\!\!\! 
%\begin{cases} 
%\int_{\Omega_{\tau}}  P \varphi(\tau,\cdot) \, dx &\!\! - \int_{\Omega_0}  P_0 \varphi(0,\cdot)dx = \nonumber\\
%&\int_0^{\tau} \int_{\Omega_t} \left( P \partial_t \varphi + P \vc{v} \cdot \Grad_x \varphi + \vc{G_P}  \varphi(t, \cdot) \right) dx dt,\nonumber \\ \\
%\int_{\Omega_{\tau}}  Q \varphi(\tau,\cdot) \, dx & \!\! - \int_{\Omega_0}  Q_0 \varphi(0,\cdot)dx  = 
%\nonumber \\
%&\int_0^{\tau} \int_{\Omega_t} \left( Q \partial_t \varphi + P \vc{v} \cdot \Grad_x \varphi + \vc{G_Q}  \varphi(t, \cdot) \right) dx dt, \nonumber \\ \\
% \int_{\Omega_{\tau}} D \varphi(\tau,\cdot) \, dx & \!\! - \int_{\Omega_0}  D_0 \varphi(0,\cdot)dx  = \nonumber\\
% &\int_0^{\tau} \int_{\Omega_t} \left( D \partial_t \varphi + D \vc{v} \cdot \Grad_x \varphi + \vc{G_D}  \varphi(t, \cdot) \right) dx dt \nonumber.
%\end{cases}
%\end{equation}

\begin{equation} 
\left.
\begin{array}{l}
\displaystyle{\int_{\Omega_{\tau}}  P \varphi(\tau,\cdot) \, dx  - \int_{\Omega_0}  P_0 \varphi(0,\cdot)dx = }\\
\hspace{1.5cm}\displaystyle{\int_0^{\tau} \!\!\int_{\Omega_t} \left( P \partial_t \varphi + P \vc{v} \cdot \Grad_x \varphi + \vc{G_P}  \varphi(t, \cdot) \right) dx dt}, \\ \\
\displaystyle{\int_{\Omega_{\tau}}  Q \varphi(\tau,\cdot) \, dx - \int_{\Omega_0}  Q_0 \varphi(0,\cdot)dx}  = 
 \\
\hspace{1.5cm}\displaystyle{\int_0^{\tau} \!\!\int_{\Omega_t} \left( Q \partial_t \varphi + P \vc{v} \cdot \Grad_x \varphi + \vc{G_Q}  \varphi(t, \cdot) \right) dx dt,}\\ \\
 \displaystyle{\int_{\Omega_{\tau}} D \varphi(\tau,\cdot) \, dx  - \int_{\Omega_0}  D_0 \varphi(0,\cdot)dx  =} \\
 \hspace{1.5cm}\displaystyle{\int_0^{\tau} \!\!\int_{\Omega_t} \left( D \partial_t \varphi + D \vc{v} \cdot \Grad_x \varphi + \vc{G_D}  \varphi(t, \cdot) \right) dx dt}.
\end{array}
\right\}
%\tag {\bf{I}}
 \label{w-Da}
\end{equation}
%In the sequel we refer to these relations as {\text{\bf (I)}}.
In particular, 
$$P \in L^p([0,T]; \Omega_{\tau}), \,\, Q \in L^p([0,T]; \Omega_{\tau}), \,\,D \in L^p([0,T]; \Omega_{\tau}) \,\, \mbox{for all}\,\, p \ge 1. $$ 

We remark that in  the weak formulation, it is convenient that the equations \eqref{dP}-\eqref{dD} hold in the whole space $\mathbb{R}^3$ provided that the densities $(P,Q,D)$ are extended to be zero outside the tumor domain.
\smallskip

%\item 
$\bullet$ Darcy's equation \eqref{pressure2} holds in the sense of distributions, i.e., for any test function 
$\vc{\varphi} \in C^{\infty}_c(\mathbb{R}^3; \mathbb{R}^3)$ satisfying 
$$ \vc{\varphi} \cdot  \vc{n}|_{\Gamma_{\tau}} = 0\,\, \mbox{for any}\,\, \tau \in [0,T],$$ 
the following integral relation holds
\begin{equation}
\int_{\Omega_\tau} \sigma \Div \vc{\varphi} \, dx - \frac{\tilde{\mu}}{K} \vc{v} \vc{\varphi}  dx=0.
\label{w-pressure2}
\end{equation}

All quantities in \eqref{w-pressure2} are required to be integrable, so in particular, 
$$\vc{v} \in W^{1,2}(\mathbb{R}^3;\mathbb{R}^3),$$
and
$$ (\vc{v - V}) \cdot \vc{n}(\tau, \cdot)|_{\Gamma_{\tau}}=0\,\, \mbox{for a.a.}\,\, \tau \in [0,T].$$

\smallskip

%\item 
$\bullet$ $C \geq 0$ is a weak solution of \eqref{dC}, i.e.,  for any test function $\varphi \in C^{\infty}_c ([0,T)\times \mathbb{R}^3), T>0$ 
the  following integral relations hold

\[
\int_{\Omega_{\tau}}  C \varphi(\tau,\cdot) \, dx - \int_{\Omega_0}  C_0 \varphi(0,\cdot)dx  = \int_0^{\tau} \!\!\int_{\Omega_{t}}  C \partial_t \varphi dx dt   -
\]
\[
\int_{0}^{\tau} \!\!\int_{\Omega_t} \nu_1  \Grad_x C\cdot \Grad_x \varphi dx dt 
- \int_0^{\tau} \!\!\int_{\Omega_t}  \left(K_1 K_P CP + K_2 K_Q(\bar{C}-C)Q\right)  C  \varphi dx dt. 
\]
\smallskip

$\bullet$ $W \geq 0$ is a weak solution of \eqref{dW}, i.e.,  for any test function $\varphi \in C^{\infty}_c ([0,T)\times \mathbb{R}^3), T>0$ 
the  following integral relations hold

\[
\int_{\Omega_{\tau}}  W \varphi(\tau,\cdot) \, dx - \int_{\Omega_0}  W_0 \varphi(0,\cdot)dx  = \int_0^{\tau} \!\!\int_{\Omega_{t}}  W \partial_t \varphi dx dt   -
\]
\[
\int_{0}^{\tau} \!\!\int_{\Omega_t} \nu_2  \Grad_x W \cdot \Grad_x \varphi dx dt 
- \int_0^{\tau} \!\!\int_{\Omega_t}   \left(\mu_1 G_1(W) P + \mu_2 G_2(W) Q \right) Wdx dt. 
\]
\end{definition}

The main result of the article now follows. 

\begin{theorem}\label{T2.2}
Let $\Omega_0 \subset \mathbb{R}^3$ be a bounded domain of class $C^{2+\nu}.$ 
Assume that the vector field $\vc{V}$ belongs to the class
$$\vc{V} \in C^1([0,T]; C^3_c(\mathbb{R}^3; \mathbb{R}^3)),\,\,\, \Div_x \vc{V}(\tau, \cdot) =0 \,\, \mbox{for all} \,\,\, \tau \in [0, T].$$
  Let the initial data satisfy
 $$P_0 \in L^{p}(\mathbb{R}^3), \,\, Q_0 \in L^{p}(\mathbb{R}^3),\,\, D_0 \in L^{p}(\mathbb{R}^3),\,\, \mbox{for all}\,\, p\ge 1$$
 and 
$$  C_0 \in L^{2}(\mathbb{R}^3)\cap L^{\infty}(\mathbb{R}^3),\,\, W_0 \in L^{2}(\mathbb{R}^3)\cap L^{\infty}(\mathbb{R}^3), $$ 
$$with \,\, (P_0, Q_0, D_0, C_{0}, W_0) \ge 0, \,\,\,  (P_0, Q_0, D_0, C_0, W_0) \not\equiv 0,$$
$$ P_0+Q_0+D_0=\varrho_{f},\quad  (P_0, Q_0, D_0, C_0, W_0)|_{\mathbb{R}^3 \setminus \Omega_0} =0. $$
Then the problem  \eqref{S} with initial data \eqref{IC}  and boundary data \eqref{BC1}-\eqref{BC3} admits a weak solution in the sense specified in Definition \ref{D2.1}.
 \end{theorem}

\section{Global Existence of Weak Solutions to the system ${\bf S_\mu}$ }\label{S3}
As already said in Section \ref{S1}, we will prove the Theorem \ref{T2.2} by performing the vanishing viscosity limit of the system \eqref{Smu}.
Therefore we consider the system \eqref{Smu} endowed with the following initial data
\begin{equation}
\begin{cases}
& \!\!\!\!  P_{\mu}(0, \cdot) = P_{\mu 0}=P_0, \,\, Q_{\mu}(0, \cdot) = Q_{\mu 0}=Q_0, \,\, D_{\mu}(0, \cdot) =D_{\mu 0}= D_0, \\
& \!\!\!\!  C_{\mu}(0, \cdot) =C_{\mu 0}=  C_{0}\leq \bar{C}, \,\, W_{\mu}(0, \cdot) =W_{\mu 0}=  W_{0} \,\,\, \text{in} \,\,\, \Omega_0.
\end{cases}
 \label{ICmu}
\end{equation}
and the following boundary data:
\begin{equation}
 (\vc{v}_{\mu} - \vc{V}) \cdot \vc{n}|_{\Gamma_{\tau}} = 0, \,\, \mbox{for any}\,\,\, \tau \ge 0. \label{BC1mu}
 \end{equation}
%In addition, for {\em viscous} fluids, Navier proposed the boundary condition of the form
\begin{equation}
[\mathbb{S} \vc{n}]_{\mbox{tan}}|_{\Gamma_{\tau}} = 0, \label{BC2mu}
\end{equation}
\begin{equation}
C_{\mu}(x,t)|_{\Gamma_t} = 0,\,\,\, W_{\mu}(x,t)|_{\Gamma_t} = 0 \label{BC3mu}.
\end{equation} 

In this section,  we discuss briefly for completeness the global existence of weak solutions to the nonlinear system \eqref{Smu} presented in \cite{DT-DrugApplication-2015}.
The following result established in 
\cite{DT-DrugApplication-2015} will be essential in the sequel.
\begin{theorem}\label{T-P1}
Let $\Omega_0 \subset \mathbb{R}^3$ be a bounded domain of class $C^{2+\nu}$  and let 
$$\vc{V} \in C^1([0,T]; C^3_c(\mathbb{R}^3; \mathbb{R}^3))$$
 be given. Let the initial data satisfy
 $$P_{\mu 0}\in L^{p}(\mathbb{R}^3), \,\, Q_{\mu 0}\in L^{p}(\mathbb{R}^3),\,\, D_{\mu 0}\in L^{p}(\mathbb{R}^3),\,\, \mbox{for all}\,\, p\ge 1$$
 and 
$$  C_{\mu 0} \in L^{2}(\mathbb{R}^3)\cap L^{\infty}(\mathbb{R}^3),\,\, W_{\mu 0} \in L^{2}(\mathbb{R}^3)\cap L^{\infty}(\mathbb{R}^3), $$ 
$$with \,\, (P_{\mu 0}, Q_{\mu 0}, D_{\mu 0}, C_{\mu 0}, W_{\mu 0}) \ge 0, \,\,\,  (P_{\mu 0}, Q_{\mu 0}, D_{\mu 0}, C_{\mu 0}, W_{\mu 0}) \not\equiv 0,$$
$$ P_{\mu 0}+Q_{\mu 0}+D_{\mu 0}=\varrho_{f},\quad  (P_{\mu 0}, Q_{\mu 0}, D_{\mu 0}, C_{\mu 0}, W_{\mu 0})|_{\mathbb{R}^3 \setminus \Omega_0} =0. $$
Then the problem  \eqref{Smu} with initial data \eqref{ICmu}  and boundary data \eqref{BC1mu}-\eqref{BC3mu} admits a weak solution satisfying the constraint
\begin{equation}
P_{\mu} +Q_{\mu}+D_{\mu}= \vr_{f}
\label{b1}
\end{equation}
% in the sense specified in Definition \ref{D2.1}.
\end{theorem}
\begin{proof}
We present here the main ingredients of the proof of the Theorem \ref{T-P1} presented in \cite{DT-DrugApplication-2015} (in order to simplify the notations we drop the index $\mu$). 
\begin{enumerate}[\quad$\bullet$]
\item Our approach involves the construction of a suitable approximating scheme which  relies on the {\em penalization} of the boundary behavior, diffusion and viscosity in the weak formulation. The approximating scheme employs the variables $\varepsilon$ (for the penalization of the boundary behavior) and $\omega$  (for the penalization of the diffusion and viscosity).
\begin{enumerate}
\item[{\bf a.}] In the center of the approach lie the so-called 
{\em generalized penalty methods}  typically suitable for treating partial slip, free surface, contact and related boundary conditions in viscous flow analysis and simulations.
This form of boundary penalty approximation appeared  by Courant in  \cite{Courant-1956}, in the context of  slip conditions for stationary incompressible fluids by Stokes and Carrey in \cite{StokesCarey-2011}, and more recently in a series of articles (cf. \cite{DT-MixedModel-2013}, \cite{DT-VariableDensity-2014}, \cite{DT-DrugApplication-2015},\cite{FeireislNS-2011}, \cite{FeireislKNNS-2013}). 

More specifically, the boundary condition \eqref{BC1} is treated as a {\em weakly enforced constraint}, in the sense that  the variational (weak) formulation of the Brinkman equation is supplemented by a singular forcing term
\begin{equation*} 
\frac{1}{\varepsilon} \int_{\Gamma_t} (\vc{v}-\vc{V}) \cdot {\bf n} \vc{\varphi} \cdot \vc{n} dS_x,\,\,\, \varepsilon>0\,\, \mbox{small}, \label{penalty}
\end{equation*}
penalizing the normal component of the velocity on the boundary of the tumor domain.

\item[{\bf b.}] A {\em variable} shear viscosity coefficient $\mu = \mu_{\omega},$ as well as a {\em variable} diffusions $\nu_i={\nu_i}_{\omega}, i=1,2 $ with $\mu_{\omega}, {\nu_i}_{\omega}$  are introduced, with the property that they vanish outside the tumor domain and remain positive within the tumor domain. The addition, of the variable $\omega$  allows us the treat the moving domain.
\end{enumerate}

\item 
Keeping $\varepsilon$ and $\omega$ fixed, we solve the modified problem in a (bounded) reference domain $B \subset \mathbb{R}^3$ chosen in such way that 
$$\bar{\Omega}_{\tau} \subset B \,\, \mbox{for any}\,\, \tau \ge 0$$  
with the aid of a Faedo-Galerkin approximation. We refer the reader to \cite{DT-DrugApplication-2015} for the details.
The solution $\{P_{\omega,\e},\  Q_{\omega,\e},\  D_{\omega,\e}, \vc{v}_{\omega,\e}\}$ constructed satisfy the following uniform bounds:

\begin{equation}
0\leq P_{\omega,\e},  Q_{\omega,\e},  D_{\omega,\e}\leq \vr_{f}\quad \text{in $[0,T]\times B$},
\label{u1a}
\end{equation}
this entails that for any $p\geq 1$
\begin{equation}
P_{\omega,\e},\  Q_{\omega,\e},\  D_{\omega,\e} \quad \text {are uniformly bounded in $L^{p}([0,T]\times B)$}.
\label{u2a}
\end{equation}
Moreover we have the following uniform bounds for  nutrient $C_{\omega,\e}$, the drug concentration $W_{\omega,\e}$ the velocity $\vc{v}_{\omega,\e}$ and the pressure $\sigma_{\omega,\e}$\begin{equation}
C_{\omega, \e}(x,t)\in L^{\infty}([0,T]\times B).
\label{mpC}
\end{equation}
\begin{equation}
W_{\omega, \e}(x,t)\in L^{\infty}([0,T]\times B).
\label{mpW}
\end{equation}

%\begin{equation}
%\frac{\partial}{\partial t}\int_{B}\frac{1}{2}C^{2}_{\omega, \e}dx + \int_{B} {\nu_{1}}_{\omega} |\Grad_{x}C_{\omega, \e}|^{2}dx\leq c\int_{B}C^{2}_{\omega, \e}dx, 
%\label{u3}
%\end{equation}
%\begin{equation}
%\frac{\partial}{\partial t}\int_{B}\frac{1}{2}W^{2}_{\omega, \e}dx +\int_{B} {\nu_{2}}_{\omega}|\Grad_{x}W_{\omega, \e}|^{2}dx\leq c\int_{B}W^{2}_{\omega, \e}dx.
%\label{u3w}
%\end{equation}
%As a consequence of \eqref{mpC}, \eqref{mpW}, \eqref{u3}, \eqref{u3w} we get the following uniform bounds  with respect to $\varepsilon$, $\omega$.\\
\begin{equation}
\|C_{\omega,\e}\|_{L^{2}_{t}L^{2}_{x}}+ \| {\nu_{1}}_{\omega} \Grad_{x} C_{\omega,\e}\|_{L^{2}_{t}L^{2}_{x}}\leq c,
\label{bc}
\end{equation}\\
\begin{equation}
\|W_{\omega,\e}\|_{L^{2}_{t}L^{2}_{x}}+  \| {\nu_{2}}_{\omega} \Grad_{x} W_{\omega,\e}\|_{L^{2}_{t}L^{2}_{x}}\leq c,
\label{bw}
\end{equation}\\
where $L^{q}_{t}L^{p}_{x}$ stands for $L^{q}(0,T;L^{2}(B)\!)$.
In addition,
%\begin{equation}
%\Div\vc{v}_{\omega,\e}={\bf G}, \qquad \text{with ${\bf G}\in L^{\infty}(0,T;L^{p}(B))$,\quad $p>1$}.
%\label{bdiv}
%\end{equation}
%\begin{equation}
% \|\nabla\vc{v}_{\omega, \e}\|_{L^{p}_{x}}\leq c\|{\bf G}\|_{L^{p}_{x}} \qquad p>1.
%\label{bgradv}
%\end{equation}

\begin{equation}
\|\sigma_{\omega,\e} \|_{L^{\beta}_{x}}\leq c, \qquad 1<\beta\leq 2
\label{bprea}
\end{equation}

 \begin{equation}
\|\mu_{\omega} \vc{v}_{\omega,\e}\|_{L^{2}_{x}}+\|\mu_{\omega} \Grad_{x}\vc{v}_{\omega,\e}\|_{L^{2}_{x}}\leq c,
\label{bv}
\end{equation}\\
\begin{equation}
\int_{\Gamma_t} |(\vc{v}_{\omega,\e}-\vc{V}) \cdot \vc{n}|^{2}dS\leq c\varepsilon.
\label{bp}
\end{equation}

\item Letting the penalization  $\varepsilon \to 0$ for fixed $\omega > 0 $ we obtain a ``two-phase" model consisting of the {\em tumor region}  and the {\em healthy tissue} separated by impermeable boundary.  We show that the densities vanish  in part of the reference domain, specifically on $((0,T) \times B) \setminus Q_T.$ 
 The main issue is to describe the evolution of the interface $\Gamma_{\tau}.$ To that effect we employ elements from the so-called {\em level set method}.
%We refer the reader to the Appendix in Section \ref{S6} where the main points of the proof are presented. 

\item The final  result is obtained by  performing the limit $\omega \to 0.$

%\item The boundary conditions considered  here are biologically relevant as confirmed  by experimental evidence. In biological settings flow situations are observed frequently for which the standard  {\em no-slip} condition give rise to significant difficulties without capturing the essence of the problem. Specific cases in which some form of {\em slip} is essential include moving boundary domains involving {\em tumor growth}.
\end{enumerate}
\par\smallskip

%The existence theory for the barotropic Navier-Stokes
%system on {\em fixed} spatial domains in the framework of weak solutions was developed in the seminal work of Lions \cite{Lions-1998}. 
\end{proof}
\par\smallskip

\par\smallskip
\section{A Priori Estimates}\label{S4}
In this section we collect all the a priori estimates uniform in $\mu$ satisfied by the solutions of the system \eqref{Smu}. 
 Let us mention that, in the sequel, we will denote by $c$ any constant that depends on $\vr_{f}$, $\bar{C}$, $\bar{W}$, the initial data \eqref{IC} and the boundary conditions \eqref{BC2}-\eqref{BC3}.
First of all we observe that because of  the condition \eqref{b1} we get that 
\begin{equation}
0\leq P_{\mu},  Q_{\mu},  D_{\mu}\leq \vr_{f}\quad \text{in $[0,T]\times \Omega_{\tau}$},
\label{u1}
\end{equation}
from which it follows that  for any $p\geq 1$
\begin{equation}
P_{\mu},\  Q_{\mu},\  D_{\mu} \quad \text {are uniformly bounded in $L^{p}([0,T]\times \Omega_{\tau})$}.
\label{u2}
\end{equation}
By a standard application of the maximum principle to the parabolic  equations satisfied by the nutrient $C_{\mu}$ and the drug concentration $W_{\mu}$ we have that
\begin{equation}
\sup_{t\in [0,T]}\|C_{\mu}\|_{L^{\infty}(\Omega_{t})}\leq \bar{C}, \qquad \sup_{t\in [0,T]}\|W_{\mu}\|_{L^{\infty}(\Omega_{t})}\leq \bar{W}
\label{bCW}
\end{equation}
Now,  by multiplying \eqref{dC} by  $C_{\mu}$,  by integrating by parts and by taking into account \eqref{u1}, \eqref{u2}, \eqref{bCW} we get that $C_{\mu}$ satisfies the following energy estimate,
\begin{equation}
\begin{split}
\int_{\Omega_{\tau}}\frac{1}{2}|C_{\mu}|^{2}dx + \nu_{1}\int_{0}^{\tau}\!\!\int_{\Omega_{t}} & |\Grad_{x}C_{\mu}|^{2}dxdt\\
&\leq c\int_{\Omega_{\tau}}|C_{0}|^{2}dx+\int_{0}^{\tau}\!\!\int_{\Omega_{t}}|C_{\mu}|^{2}dxdt, 
\end{split}
\label{u3}
\end{equation}
similarly, taking into account that $G_{1}$ and $G_{2}$ are smooth functions we have also
\begin{equation}
\begin{split}
\int_{\Omega_{\tau}}\frac{1}{2}|W_{\mu}|^{2}dx + \nu_{2}\int_{0}^{\tau}\!\!\int_{\Omega_{t}} & |\Grad_{x}W_{\mu}|^{2}dxdt\\
&\leq c\int_{\Omega_{\tau}}|W_{0}|^{2}dx+\int_{0}^{\tau}\!\!\int_{\Omega_{t}}|W_{\mu}|^{2}dxdt, 
\end{split}
\label{u3w}
\end{equation}
As a consequence of \eqref{u3} and \eqref{u3w} we get the following uniform bounds
\begin{equation}
\int_{0}^{T}\|C_{\mu}\|^{2}_{W^{1,2}(\Omega_{t})}dt\leq c, \qquad \int_{0}^{T}\|W_{\mu}\|^{2}_{W^{1,2}(\Omega_{t})}dt\leq c.
\label{u4}
\end{equation}
Now we focus our attention on the velocity field $\vc{v}_{\mu}$. First we notice that by adding up the equations $\eqref{Smu}_{2}-\eqref{Smu}_{4}$ we have 
\begin{equation}
\rho_{f}\Div{\vc{v}_{\mu}}= K_{B} C_{\mu} P_{\mu}-K_{R}D_{\mu}=\vc{G},
\label{div1}
\end{equation}
where by using \eqref{u2} we have that $\vc{G}\in L^{p}(\Omega_{\tau})$, $p\geq 1$. 
Next, by applying regularity theory concerning the divergence  equation in Sobolev spaces (see Lemma 2.1.1 (a) in \cite{Sohr-2001} or Remark 3.19 in \cite{Novotny-Stras-2004}, for more details see also \cite{DT-MixedModel-2013}) we end up with 
\begin{equation}
\|\Grad_x\vc{v}_{\mu}\|_{L^{p}_{x}}\leq c\|{\bf G}\|_{L^{p}_{x}}, \qquad p>1.
\label{bgradv}
\end{equation}
On the other hand by considering the equation $\eqref{Smu}_{1}$, by taking into account \eqref{div1} and \eqref{bgradv} and by a standard application of elliptic regularity theory (see again \cite{DT-MixedModel-2013}) we conclude with the following uniform bound with respect to $\mu$,
\begin{equation}
\|\sigma_{\mu} \|_{L^{2}_{x}}\leq c.
\label{bpre}
\end{equation}
Now, by using \eqref{div1}, \eqref{bpre} and by multiplying the equation $\eqref{Smu}_{1}$ by $\vc{v}_{\mu}$ and by integrating by parts we have
\begin{equation}
\frac{\tilde{\mu}}{K}\int_{\Omega_{\tau}}|\vc{v}_{\mu}|^{2}dx+\mu\int_{\Omega_{\tau}}|\Grad_x\vc{v}_{\mu}|^{2}dx\leq c.
\label{u5}
\end{equation}
We remark  that the soleindal condition \eqref{incomp-bc} on $\vc{V}$ was essential in order to get the estimates \eqref{u3}, \eqref{u3w}, \eqref{u5}.

\section{Vanishing viscosity limit $\mu\to 0$}\label{S5}
In this section we perform the limit $\mu\to 0$ in order to recover the system \eqref{S}. Since our limiting process takes place in a moving domain $\Omega_{\tau}$ it is more convenient to extend $\vc{v}_{\mu}$, $C_{\mu}$ and $W_{\mu}$, $P_{\mu}$,  $Q_{\mu}$, $D_{\mu}$ on the whole domain $\R^{3}$ by  setting them equal to zero outside the tumor domain. In fact in this way one performs the limiti in a ``time independent domain''. Then, since the domain $\Omega_{\tau}$ is regular at each time we  use the standard extension operator for Sobolev spaces, $E:W^{1,2}(\Omega_{\tau})\longrightarrow W^{1,2}(\R^{3})$, uniformely bounded with respect to $t\in [0,T]$, (for details on the operator E see \cite{AF2003}). 

From the uniform bounds \eqref{u4}, \eqref{u5} we get
\begin{align}
EC_{\mu}\longrightarrow C\qquad  &\text{weakly in $L^{2}(0,T;W^{1,2}(\R^{3}))$}, \label{cc}\\
EW_{\mu}\longrightarrow W\qquad &\text{weakly in $L^{2}(0,T;W^{1,2}(\R^{3}))$}, \label{cw}\\
E\vc{v}_{\mu}\longrightarrow \vc{v} \qquad &\text{weakly in $W^{1,2}(\R^{3})$}. \label{cv}
\end{align}
By taking into account \eqref{u2} and \eqref{u5} we have 
we get 
\begin{equation}
\hspace{-0.55cm}
\left.
\begin{array}{r}
P_{\mu}\vc{v}_{\mu}\rightarrow P\vc{v}\\ \\
Q_{\mu}\vc{v}_{\mu}\rightarrow Q\vc{v}\\ \\
D_{\mu}\vc{v}_{\mu}\rightarrow D\vc{v}
\end{array}
\right\}\  \text{weakly-($\ast$) in} \ L^{\infty}(T_{1},T_{2};L^{2q/q+2}(K)),\  2\leq q< 6,.
\label{cPQDCv}
\end{equation}
where $K\subset\Omega_{\tau}$ is a compact subset.  Moreover, from the equations $\eqref{Smu}_{2}-\eqref{Smu}_{4}$ it follows that

\begin{equation}
\hspace{-0.55cm}
\left.
\begin{array}{r}
P_{\mu}\vc{v}_{\mu}\rightarrow P\vc{v}\\ \\
Q_{\mu}\vc{v}_{\mu}\rightarrow Q\vc{v}\\ \\
D_{\mu}\vc{v}_{\mu}\rightarrow D\vc{v}
\end{array}
\right\}\quad\text{in} \ C_{\text{weak}}([T_{1},T_{2}];L^{2q/q+2}(K)),\  2\leq q< 6.
\label{cPQDCv2}
\end{equation}
Now, by using \eqref{u2}, \eqref{u4}, as before we get also
\begin{equation}
\hspace{-0.55cm}
\left.
\begin{array}{r}
P_{\mu}C_{\mu}\rightarrow PC\\ \\
Q_{\mu}C_{\mu}\rightarrow QC\\ \\
D_{\mu}C_{\mu}\rightarrow DC\\ \\
P_{\mu}W_{\mu}\rightarrow PW\\ \\
Q_{\mu}W_{\mu}\rightarrow QW
\end{array}
\right\}\ \text{weakly-($\ast$) in $L^{\infty}(0,T;L^{2q/q+2}(K)),$ $2\leq q < 6$.}
\label{cC}
\end{equation}
\begin{remark}
Since the compact set $K$ can be chosen arbitrarily close to the the boundary of $Q_{T}$, the above convergences \eqref{cPQDCv}, \eqref{cPQDCv2}, \eqref{cC} take place
in the whole cylinder $Q_{T}$.
\end{remark}
Now, by standard computations, and by taking into account \eqref{u5} we have
\begin{equation*}
\sqrt{\mu}\int_{\Omega_{\tau}}
\sqrt{\mu}\Grad_x \vc{v}_{\omega}  : \Grad_x \vc{\varphi}  dx\longrightarrow 0 
\quad \text{as $\mu\to 0$}.
%\label{visc5}
\end{equation*}
At this point it is straightforward to pass into the limit in the weak formulations of the system \eqref{Smu}   and to conclude the proof of the Theorem \ref{T2.2}.

\section{Acknowlegments}
The work of D.D. was supported  by the Ministry of Education, University and Research (MIUR), Italy under the grant PRIN 2012- Project N. 2012L5WXHJ, \emph{Nonlinear Hyperbolic Partial Differential Equations, Dispersive and Transport Equations: theoretical and applicative aspects}. Ê K.T.  gratefully acknowledges the support  in part by the National Science Foundation under the grant DMS-1211519 and by the Simons Foundation under the Simons Fellows in Mathematics Award 267399.

\end{document}